\documentclass{amsart}
\usepackage{amssymb, mathrsfs, amsmath}

\title[Degree gaps for multipliers]{Degree gaps for multipliers and the dynamical Andr\'{e}-Oort conjecture}
\author{Patrick Ingram}
\address{Department of Mathematics and Statistics, York University}

\newcommand{\QQ}{\mathbb{Q}}
\newcommand{\ZZ}{\mathbb{Z}}
\newcommand{\CC}{\mathbb{C}}
\newcommand{\RR}{\mathbb{R}}

\newcommand{\PP}{\mathbb{P}}

\newcommand{\ord}{\operatorname{ord}}

\renewcommand{\epsilon}{\varepsilon}

\newtheorem*{theorem}{Theorem}
\newtheorem*{corollary}{Corollary}
\newtheorem{lemma}{Lemma}
\newtheorem{prop}[lemma]{Proposition}

\theoremstyle{remark}
\newtheorem{remark}[lemma]{Remark}

\theoremstyle{definition}

\begin{document}
\begin{abstract}
We demonstrate how recent work of Favre and Gauthier, together with a modification of a result of the author, shows that a family of polynomials with infinitely many post-critically finite specializations cannot have any periodic cycles with multiplier of very low degree, except those which vanish, generalizing results of Baker and DeMarco, and Favre and Gauthier.
\end{abstract}

\maketitle

\section{introduction}

Let $f(z)\in\CC[z]$ have degree $d\geq 2$. Much about the behaviour of $f$ under iteration can be gleaned from considering just the orbits of the critical points, and so \emph{post-critically finite (PCF)} polynomials, those whose critical orbits are all finite, have garnered particular attention. A dimension count suggests the heuristic that, while PCF polynomials ought to be Zariski-dense in the space of all polynomials of given degree, there ought not exist  algebraic families of positive dimension (modulo change-of-variables), and indeed this is known to be true. The general philosophy of \emph{unlikely intersections} then suggests that any subvariety of this space in which the PCF maps are Zariski-dense must be defined by conditions on critical orbits, a claim now known as the \emph{Dynamical Andr\'{e}-Oort Conjecture}. Several deep results in this direction have been established in recent years~\cite{bd, dwy, fgcubic,   ght, gkn, gkny, gy}, but one of the earliest is of particular interest here.

For fixed $\lambda\in\CC$, Baker and DeMarco~\cite{bd} considered the curve in the space of cubic polynomials cut out by the existence of a fixed point $f(P)=P$ with \emph{multiplier} $\lambda_f(P):=f'(P)$ equal to $\lambda$. This is a condition on critical orbits if and only if $\lambda=0$ (in which case it requires that some critical point is fixed), and Baker and DeMarco showed that this is also exactly the curve in this family on which one sees infinitely many PCF points. This result was extended to cubic polynomials with marked periodic points by Favre and Gauthier~\cite{fgcubic}, and to quadratic rational maps by DeMarco, Wang, and Ye~\cite{dwy}. Generalizing, one might suspect that any one-parameter family of polynomials with a marked periodic point of multiplier $\lambda$ witnesses infinitely many PCF specializations only if $\lambda =0$. This is our main result, not just in the case that $\lambda$ is constant, but even if it varies with low degree as a function on the parametrizing curve.

\begin{theorem}
Let $f$ be a family of polynomials parametrized over a curve $X$, with a marked periodic point $P$ of multiplier $\lambda_f(P)\in\CC(X)$, and suppose that $f_t$ is PCF for infinitely many $t\in X(\CC)$. If $\deg(\lambda_f(P))<h_{\mathrm{crit}}(f)$, then $\lambda_f(P)= 0$ identically on $X$.
\end{theorem}

The quantity $h_{\mathrm{crit}}(f)\geq 0$ is a measure of complexity of the generic critical orbits defined below, whose vanishing exactly coincides with isotriviality of the family. For any non-isotrivial family, we have
\begin{equation}\label{eq:range}0\leq \frac{\deg(\lambda_f(P))}{h_\mathrm{crit}(f)}\leq d-1,\end{equation}
and every rational number in this range is realized by some example (see Proposition~\ref{prop:range}). In other words, our main result is that  a gap appears at the left-hand-side of this range once one restricts attention to families with infinitely many PCF specializations, with the ratio never falling below 1 (excepting examples with $\lambda_f(P)\equiv 0$).  While we do not know if the Theorem is sharp, we show in Remark~\ref{rem:sharp} that the statement of the Theorem becomes false if, for any $\epsilon>0$, we replace the condition that $\deg(\lambda_f(P))<h_{\mathrm{crit}}(f)$ by the weaker condition that $\deg(\lambda_f(P))<(2+\epsilon)h_{\mathrm{crit}}(f)$.

The Theorem follows from recent work of Favre and Gauthier~\cite{fg}, and a modification of contributions of the author~\cite{pern}. It implies a generalization of the aforementioned results of Baker and DeMarco~\cite{bd} and Favre and Gauthier~\cite{fgcubic} on families of cubic polynomials with marked periodic point of constant multiplier. 

%
%Now let $f$ be a family of polynomials whose coefficients are rational functions on some curve $X/\CC$. To this family we may assign a measure of complexity of the generic critical orbits, $h_{\mathrm{crit}}(f)\geq 0$, which vanishes just in case the family is isotrivial.
%In general, if $f$ is non-isotrivial, and if $P$ is periodic for $f$ with multiplier $\lambda_f(P)\in\CC(X)$ then we have
%\[0\leq \frac{\deg(\lambda_f(P))}{h_\mathrm{crit}(f)}\leq d-1,\]
%and every rational number in this range is thus realized (see below).
%From recent work of Favre and Gauthier~\cite{fg}, and a modification of contributions of the author~\cite{pern}, we derive the following result, which shows that for families with infinitely many post-critically finite (PCF) specializations (and in contrast with the general case) there is a gap in degrees of multipliers.
%\begin{theorem}
%Let $f$ be a family of polynomials parametrized over a curve $X$, with a periodic point $P$ of multiplier $\lambda_f(P)$, and suppose that $f_t$ is PCF for infinitely many $t\in X(\CC)$. Then $\lambda_f(P)= 0$ identically on $X$ or else
%\[\deg(\lambda_f(P))\geq h_{\mathrm{crit}}(f).\]
%\end{theorem}

\begin{corollary}
Let $\lambda\in \CC$, and let $f$ be a non-isotrivial family of polynomials with a periodic point of constant multiplier $\lambda\in\CC$, parametrized over a curve $X$. If $f_t$ is PCF for infinitely many $t\in X(\CC)$, then $\lambda=0$.
\end{corollary}

The  reader will have noted that this corollary generalizes only one direction of the results for cubic polynomials, but it is the non-trivial direction, and in general the converse to the corollary is false (see Remark~\ref{rem:contra}).

The theorem is proved by applying the main step in the proof of~\cite[Theorem~G]{fg} of Favre and Gauthier (whom the author thanks for comments on an earlier version of this note), which shows that if $f$ has infinitely many PCF specializations, then $f$ has, on the generic fibre and  up to a natural equivalence, at most one infinite critical orbit. The author showed in~\cite{pern} that a non-isotrivial family with a marked periodic point of multiplier $\lambda$ of sufficiently large degree has at least two independent, infinite critical orbits, but the notion of dependence used there was not the same, and some work is required to combine the results.

\section{Notation and conventions}\label{sec:notation}

Let $F$ be a field, and let $|\cdot|$ be an absolute value on the algebraic closure $\overline{F}$.
For $f(z)\in F[z]$ of degree $d\geq 2$, a \emph{point of period $n$} is a solution to $f^n(Q)=Q$ (with $n$ minimal), and the \emph{multiplier} of this point is
\[\lambda_f(Q)=(f^n)'(Q)=\prod_{i=0}^{n-1}f'(f^i(Q)).\]
For $P\in\overline{F}$, we  define as usual
\[G_f(P)=\lim_{n\to\infty}d^{-n}\log^+|f^n(P)|,\]
where $\log^+ x= \log\max\{1, x\}$ for $x\in\RR$.
This limit  always exists, and vanishes at preperiodic points.
We will set
\[g_{\mathrm{crit}}(f)=\max_{f'(c)=0}G_f(c),\]
where we take solutions in the algebraic closure.
It is straightforward to show that $g_{\mathrm{crit}}(f)$ is independent of coordinates, in that if $\phi$ is an affine transformation with coefficients in $\overline{F}$,
\[g_{\mathrm{crit}}(\phi^{-1}\circ f\circ \phi) = g_{\mathrm{crit}}(f). \]

For most of our argument, we will restrict attention to a certain normal form for polynomials, namely
\begin{align}\label{eq:form}
f_{\mathbf{c}}(z)&=\frac{1}{d}z^d-\frac{1}{d-1}(c_1+\cdots +c_{d-1})z^{d-1}+\cdots \pm c_1c_2\cdots c_{d-1}z\\
&=\sum_{i=1}^d \frac{(-1)^{d-i}}{i}\sigma_{d-i, d-1}(\mathbf{c})z^i\nonumber
\end{align}
just as in~\cite{pcfpoly}, where $\sigma_{j, k}$ is the fundamental symmetric polynomial of degree $k$ in $j$ variables. Note that $f_{\mathbf{c}}(0)=0$, and \[f'_\mathbf{c}(z)=(z-c_1)\cdots (z-c_{d-1}),\] so that $c_1, ..., c_{d-1}$ are the critical points of $f_{\mathbf{c}}$.

Now let $X/\CC$ be a curve, assumed without loss of generality to be smooth and projective, and let $|\cdot|$ be a non-trivial absolute value on the  function field $\CC(X)$ which is trivial on $\CC$. These all have the form
$|a|=e^{-C\ord_{z=v}(a)}$
for some $C>0$ and $v\in X(\CC)$. We will normalize these by taking $C=1$, and so to each $v\in X(\CC)$ we identify the absolute value
\[|a|_v=e^{-\ord_{z=v}(a)}.\]
%Any finite extension of $E/\CC(X)$ is the function field of a smooth projective cover $Y\to X$ of degree $m=[E:\CC(X)]$, and we can extend $|\cdot|_v$ to $E$ by taking
%\[|a|_v=|a|_w^{r/m}\]
%where $w\in Y(\CC)$ is a point mapping to $v\in X(\CC)$ with local degree $r$. 
It is possible to extend $|\cdot|_v$ (non-uniquely) to the algebraic closure of $\CC(X)$.

 Note that for $a\in\CC(X)$ non-zero,
\begin{equation}\label{eq:prodfla}\sum_{v\in X(\CC)}\log|a|_v=0\end{equation}
and
\begin{equation}\label{eq:height}\sum_{v\in X(\CC)}\log^+|a|_v=\deg(a).\end{equation}
In this setting, degree is the appropriate notion of \emph{height}, and for a tuple $\mathbf{a}=(a_1, ..., a_k)\in\CC(X)^k$ we will set
\begin{equation}\label{eq:heighttuple}h(\mathbf{a})=\sum_{v\in X(\CC)}\log^+\|\mathbf{a}\|_v.\end{equation}
Notice that $h(\mathbf{a})=0$ if and only if all $a_i$ are constant. The \emph{critical height} of a polynomial $f(z)$ with coefficients in $\CC(X)$ is defined here as
\[h_{\mathrm{crit}}(f)=\sum_{v\in X(\CC)}g_{\mathrm{crit}, v}(f),\]
where $g_{\mathrm{crit}, v}$ is defined as above relative to the absolute value $|\cdot|_v$.
Note that the critical height is independent of the choice of coordinates. This is not the same critical height $\hat{h}_{\mathrm{crit}}(f)$ as used in~\cite{pcfpoly, hcrit, barbados}, but one sees from the non-negativity of $G_{f}$ that
\[h_{\mathrm{crit}}(f)\leq \hat{h}_{\mathrm{crit}}(f)\leq (d-1)h_{\mathrm{crit}}(f).\]

Finally, note that if $\pi:Y\to X$ is a finite branched cover, then $\pi^*$ gives an embedding of $\CC(X)$ into $\CC(Y)$, and
\begin{equation}\label{eq:localdeg}|\pi^*a|_w=|a|_{\pi(w)}^{e_\pi(a)},\end{equation}
where $e_{\pi}(a)$ is the degree of $\pi$ locally at $a$.
If $\mathbf{a}\in\CC(X)^k$, we may pull back the components to obtain a tuple $\pi^*\mathbf{a}\in\pi^*\CC(X)^k\subseteq \CC(Y)^k$, and it is easy to check from the definitions above that
\[h(\pi^*\mathbf{a})=\deg(\pi)h(\mathbf{a}).\]
Similarly, if
$f(z)\in \CC(X)[z]$, we may pull-back coefficients to obtain a polynomial $\pi^*f\in \CC(Y)$ and we have
\begin{equation}\label{eq:hcritpull}h_\mathrm{crit}(\pi^*f)=\deg(\pi)h_{\mathrm{crit}}(f).\end{equation}
Finally note that since any finite extension of $\CC(X)$ is of the form $\CC(Y)$ as above, \eqref{eq:localdeg} allows us to choose an extension of any absolute value on $\CC(X)$ to any finite extension, and hence to the algebraic closure, as claimed above.

\section{Local lemmas}

In this section, we work over $\CC(X)$ relative to some fixed absolute value $|\cdot|$. We note, though, for interest, that we may replace $\CC(X)$ with any field of characteristic 0 or greater than $d$, with a non-archimedean, non-$p$-adic absolute value $|\cdot|$, so that 
\[|x+y|\leq \max\{|x|, |y|\}\]
for all $x, y\in F$, and $|m|=1$ for any non-zero integer $m$. This setting gives us a particularly strong variant of the results in~\cite{pcfpoly}.

\begin{lemma}\label{lem:gcrit}
For $\mathbf{c}\in \CC(X)^{d-1}$  and $f_\mathbf{c}$ as in~\eqref{eq:form}, we have
\[g_{\mathrm{crit}}(f_\mathbf{c})=\log^+\|\mathbf{c}\|.\]
\end{lemma}

\begin{proof}
This follows from~\cite{pcfpoly}, but briefly we have $f^n_\mathbf{c}(c_i)\in \CC[\mathbf{c}]$ of degree $d^n$, so $d^{-n}\log^+|f_\mathbf{c}(c_i)|\leq \log^+\|\mathbf{c}\|$, from which $G_{f_{\mathbf{c}}}(c_i)\leq \log^+\|\mathbf{c}\|$, establishing an inequality in one direction.

 In the other direction, it is shown in~\cite{pcfpoly} that the homogeneous forms $f_{\mathbf{c}}(c_i)$ (for $1\leq i\leq d-1$) have no common root, and by Hilbert's Nullstellensatz there  is an $i$ with $\log\|f_{\mathbf{c}}(c_i)\|=d\log\|\mathbf{c}\|$. If $\log\|\mathbf{c}\|>0$ then by induction we have $\log^+|f_\mathbf{c}^n(c_i)|=d^n\log^+\|\mathbf{c}\|$, giving a bound in the opposite direction. 
In the case $\log\|\mathbf{c}\|\leq 0$ this direction is trivial by the non-negativity of $G_f$.
\end{proof}

The following is a stronger form of~\cite[Lemma~7]{pern}.

\begin{lemma}\label{lem:localindep}
If $\log|c_1|<\log\|\mathbf{c}\|$, then either $\log^+\|\mathbf{c}\|=0$ or else there exists an index $i\neq 1$ with
\[G_{f_{\mathbf{c}}}(c_i)>G_{f_{\mathbf{c}}}(c_1).\]
\end{lemma}

\begin{proof}
Suppose that $\log|c_1|<\log\|\mathbf{c}\|$ and that $\log^+\|\mathbf{c}\|\neq 0$, or in other words $\log\|\mathbf{c}\|>0$. Choose $0<\epsilon<1$ so that $\log|c_1|\leq(1-\epsilon)\log\|\mathbf{c}\|$.

First, note that $f_{\mathbf{c}}(c_1)\in c_1^2\CC[c_1, ..., c_{d-1}]$, and so
\[\log|f_{\mathbf{c}}(c_1)|\leq 2\log|c_1|+(d-2)\log^+\|\mathbf{c}\|\leq (d-2\epsilon)\log^+\|\mathbf{c}\|.\]
Now, suppose that $\log|f^k_{\mathbf{c}}(c_1)|\leq (d-2\epsilon)d^{k-1}\log^+\|\mathbf{c}\|$. For $\sigma_i$ the degree-$i$ symmetric function in $d-1$ variables, we then have for $i< d-1$
\begin{align*}
\log\left|\frac{1}{i}\sigma_{d-i}(c_1, ..., c_{d-1})(f^k_{\mathbf{c}}(c_1))^i\right|&\leq (d-i)\log^+\|\mathbf{c}\|+i\log|f_{\mathbf{c}}^k(c_1)|\\
	&< (d-i)\log^+\|\mathbf{c}\|+i(d-2\epsilon)d^{k-1}\log^+\|\mathbf{c}\|\\
&< (d-2\epsilon)d^{k}\log^+\|\mathbf{c}\|
\end{align*}
just because $1< (d-2\epsilon)d^{k-1}$. So
\[\begin{split}
d^{-(k+1)}\log|f_{\mathbf{c}}^{k+1}(c_1)|&\leq \max_{1\leq i\leq d-1}\left\{d\log\left|f_{\mathbf{c}}^{k}(c_1)\right|, \log\left|\frac{1}{i}\sigma_{d-i}(c_1, ..., c_{d-1})(f^k_{\mathbf{c}}(c_1))^i\right|\right\}\\
&< \left(1-\frac{2\epsilon}{d}\right)\log^+\|\mathbf{c}\|,
\end{split}\]
and by induction \[G_{f_\mathbf{c}}(c_1)\leq\left(1-\frac{2\epsilon}{d}\right)\log^+\|\mathbf{c}\|< \log^+\|\mathbf{c}\|.\]
But we have already seen that there is some index $i$ with $G_{f_\mathbf{c}}(c_i)=\log^+\|\mathbf{c}\|$.
\end{proof}

\section{Proof of the Theorem}

We maintain the notation of the last section, but now vary the absolute value on $\CC(X)$. For each $v\in X(\CC)$, quantities from the previous sections relative to the absolute value $|\cdot|_v$ defined in Section~\ref{sec:notation} acquire a subscript $v$.

Note that by Lemma~\ref{lem:gcrit}, we have 
\begin{equation}\label{eq:hcrit} h_{\mathrm{crit}}(f_{\mathbf{c}})=\sum_{v\in X(\CC)}g_{\mathrm{crit}, v}(f_\mathbf{c})=\sum_{v\in X(\CC)}\log^+\|\mathbf{c}\|_v=h(\mathbf{c}),\end{equation}
and so $h_{\mathrm{crit}}(f_{\mathbf{c}})=0$ if and only if the $c_i$, and hence coefficients of $f_{\mathbf{c}}$, are all constant.

 On the other hand, for any polynomial $f(z)\in \CC(X)$ there is a finite extension $\CC(Y)/\CC(X)$ and a map $\phi(z)=az+b$ defined over $\CC(Y)$ with $\phi^{-1}\circ f\circ \phi$ in the form~\eqref{eq:form}. Then from~\eqref{eq:hcritpull} we have $h_{\mathrm{crit}}(f)=0$ if and only if $\phi^{-1}\circ f\circ \phi$ has constant coefficients, showing that the vanishing of $h_{\mathrm{crit}}$ is equivalent to isotriviality.

The following lemma shows that in the normal form~\eqref{eq:form}, quantity $h_{\mathrm{crit}}(f)$ can be estimated using information from only a subset of the points of $X$, and is a variant of~\cite[Lemma~9]{pern}. We define
\begin{equation}\label{eq:Sdef}S=S(\mathbf{c})=\left\{v\in X(\CC):\log|c_1|_v<\log\|\mathbf{c}\|_v\right\}.\end{equation}
In other words, $S$ is the finite set of places of $X$ at which one of the functions $c_2/c_1, c_3/c_1, ..., c_{d-1}/c_1$ has a pole. Intuitively, as we approach a point of $S$ on $X(\CC)$, some other critical point becomes arbitrarily-much larger than the critical point marked by $c_1$, but the reader will also notice that $S$ is precisely the set of places at which the condition in Lemma~\ref{lem:localindep} is met.

\begin{lemma}\label{lem:Sweight}
If $\lambda_{f_{\mathbf{c}}}(0)\neq 0$, then 
\[(d-1)\sum_{v\in S}\log^+\|\mathbf{c}\|_v\geq h_{\mathrm{crit}}(f_{\mathbf{c}})-\deg(\lambda_{f_{\mathbf{c}}}(0)).\]
\end{lemma}

\begin{proof}
Since $\prod_{i= 1}^{d-1} c_i=(-1)^{d-1}\lambda_{f_{\mathbf{c}}}(0)$, we have both
\begin{equation}\label{eq:loglogplus}\log\|\mathbf{c}\|_v\leq \log^+\|\mathbf{c}\|_v\leq \log\|\mathbf{c}\|_v+\frac{1}{d-1}\log^+|\lambda_{f_{\mathbf{c}}}(0)^{-1}|_v\end{equation}
and
\begin{equation}\label{eq:c1bound}|c_1|_v^{-1}=|\lambda_{f_{\mathbf{c}}}(0)^{-1}|_v\prod_{i=2}^{d-1}|c_i|_v\leq |\lambda_{f_{\mathbf{c}}}(0)^{-1}|_v\|\mathbf{c}\|_v^{d-2}.\end{equation}
We apply~\eqref{eq:loglogplus} to obtain
\begin{align*}
\sum_{v\not\in S}\log^+\|\mathbf{c}\|_v&\leq  \sum_{v\not\in S}\left(\log\|\mathbf{c}\|_v+\frac{1}{d-1}\log^+|\lambda_{f_{\mathbf{c}}}(0)^{-1}|_v\right)\\ 
&\leq \sum_{v\not\in S}\log|c_1|_v+\sum_{v\not\in S}\frac{1}{d-1}\log^+|\lambda_{f_{\mathbf{c}}}(0)^{-1}|_v,\\
\intertext{by the definition of $S$,}
&= \sum_{v\in S}\log|c_1|^{-1}_v + \sum_{v\not\in S}\frac{1}{d-1}\log^+|\lambda_{f_{\mathbf{c}}}(0)^{-1}|_v,\\
\intertext{by the product formula~\eqref{eq:prodfla},} 
&\leq \sum_{v\in S}\left(\log^+|\lambda_{f_{\mathbf{c}}}(0)^{-1}|_v+(d-2)\log^+\|\mathbf{c}\|_v\right)\\&\quad +\sum_{v\not\in S}\frac{1}{d-1}\log^+|\lambda_{f_{\mathbf{c}}}(0)^{-1}|_v,\\
\intertext{by~\eqref{eq:c1bound}}
&\leq (d-2)\sum_{v\in S}\log^+\|\mathbf{c}\|_v+\sum_{v\in X(\CC)}\log^+|\lambda_{f_{\mathbf{c}}}(0)^{-1}|_v\\
&=  (d-2)h(\mathbf{c})-(d-2)\sum_{v\not\in S}\log^+\|\mathbf{c}\|_v +  \deg(\lambda_{f_{\mathbf{c}}}(0)),
\end{align*}
by~\eqref{eq:heighttuple}. 
But then
\[(d-1)\sum_{v\not\in S}\log^+\|\mathbf{c}\|_v\leq (d-2)h(\mathbf{c})+\deg(\lambda_{f_{\mathbf{c}}}(0)),\]
from which
\begin{align*}
(d-1)\sum_{v\in S}\log^+\|\mathbf{c}\|_v&= (d-1)h(\mathbf{c})- (d-1)\sum_{v\not\in S}\log^+\|\mathbf{c}\|_v\\
&\geq h(\mathbf{c})-\deg(\lambda_{f_{\mathbf{c}}}(0)),
\end{align*}
which by~\eqref{eq:hcrit} is what was claimed. 
\end{proof}

\begin{proof}[Proof of the Theorem]
First, some reductions. We claim that $h_{\mathrm{crit}}(f^n)=h_{\mathrm{crit}}(f)$ for all $n$ (which differs from~\cite{hcrit} because of our slightly different definition of the critical height here). To see this, note that for any $v\in X(\CC)$ we have, by the chain rule and the observation $G_{f^n, v}=G_{f, v}$, that
\begin{align*}
	g_{\mathrm{crit}, v}(f^n)&=\max_{(f^n)'(c)=0}G_{f^n, v}(c)\\
	&=\max_{\substack{f'(f^i(c))=0\\ 0\leq i<n}}G_{f, v}(c)\\
	&=\max_{f'(\zeta)=0}\max_{\substack{f^i(c)=\zeta\\0\leq i<n}}G_{f, v}(c)\\
	&=\max_{f'(\zeta)=0}G_{f, v}(\zeta)\\&=g_{\mathrm{crit}, v}(f),	
\end{align*}
since $f^i(c)=\zeta$ implies $G_{f, v}(c)=d^{-i}G_{f, v}(\zeta)$. Summing over all $v\in X(\CC)$ establishes the claim.
Also, if $P$ is a point of period $n$ and multiplier $\lambda$ for $f$, then $P$ is a fixed point of multiplier $\lambda$ for $f^n$ (by definition). Finally, $f_t$ is PCF if and only if $f_t^n$ is, and so if the claim in the theorem is true for fixed points, it is true for periodic points. We henceforth assume $n=1$, that is, that $f(P)=P$.

Furthermore, the statement of the theorem is preserved under passing to a finite branched cover $\pi:Y\to X$. Specifically, if $f$ and $\lambda$ are already defined over $\CC(X)$, then write  $\pi^*f$ as above for the polynomial obtained by pulling-back the coefficients of $f$ to $Y$. We have $\deg(\pi^*\lambda)=\deg(\pi)\deg(\lambda)$, while $h_{\mathrm{crit}}(\pi^*f)=\deg(\pi)h_{\mathrm{crit}}(f)$, and so the ratio $\deg(\lambda_f(P))/h_{\mathrm{crit}}(f)$ is unchanged. Since the statement of the theorem is also coordinate-free, we may  freely change variables. Passing to a finite extension and choosing a new coordinate we may then assume, without loss of generality, that $f$ has form~\eqref{eq:form} with $c_i\in \CC(X)$, and that $P=0$ is the fixed point in question.

Suppose that $f$ has infinitely many PCF specializations, and for each critical point $c$ write
\[D(f, c)=\sum_{v\in X(\CC)}G_{f,v}(c)[v]\in\operatorname{Div}(X)\otimes \QQ\] as in \cite[Equation~8]{var}, 
which is the same as $\mathsf{D}_{f, c}$ in~\cite[Definition~4.6, p.~114]{fg}. Given any two critical points, neither one generically preperiodic, we have by~\cite[Theorem~37, p.~135]{fg} that the corresponding divisors are proportional. So there is a single divisor $D$ on $X$ such that for each $f'(c)=0$, we have $D(f, c)=\alpha(f, c) D$ for some $\alpha(f, c)\in\QQ$, and in at least one case $\alpha(f, c)\neq 0$. Writing $f$ in the normal form~\eqref{eq:form} we can, by permuting the $c_i$, take $0\leq \alpha(f, c_i)\leq \alpha(f, c_1)\neq 0$ for all $i$. So in particular \[G_{f, v}(c_1)\geq G_{f, v}(c_i)\] for any $1\leq i\leq d-1$ and any $v\in X(\CC)$.
Let $S$ be the set of points of $X$ defined in~\eqref{eq:Sdef}. By Lemma~\ref{lem:localindep}, if $v\in S$, then we must have $\log^+\|\mathbf{c}\|=0$. But by Lemma~\ref{lem:Sweight}, if $\lambda_{f_{\mathbf{c}}}(0)\neq 0$, we now have
\[0=(d-1)\sum_{v\in S}\log^+\|\mathbf{c}\|_v\geq h_{\mathrm{crit}}(f_{\mathbf{c}})-\deg(\lambda_{f_{\mathbf{c}}}(0)).\]
Since all cases were reduced to this one, this proves the theorem in general.
\end{proof}

The Theorem confirmed, we now demonstrate that the ratio $\deg(\lambda_f(P))/h_{\mathrm{crit}}(f)$ really is constrained as in~\eqref{eq:range}.
\begin{prop}\label{prop:range} 
	 If $P\in\CC(X)$ is periodic for $f$, then
	\[\deg(\lambda_f(P))\leq (d-1)h_{\mathrm{crit}}(f).\]
	Furthermore, the ratio $\deg(\lambda_f(P))/h_{\mathrm{crit}}(f)$ can take any rational value in the interval $[0, d-1]$.
\end{prop}

\begin{proof}
We claim, in fact, that in each absolute value $v$ we have
	\[\log^+|\lambda_f(P)|_v\leq (d-1)g_{\mathrm{crit}, v}(f).\]
The proposition is then proved by summing over all $v\in X(\CC)$.

Note that if $P$ has period $n$, we can replace $f$ by $f^n$ and assume that $P$ is fixed. Also, since both sides are coordinate independent, we may without loss of generality replace $X$ by a finite cover, and change coordinates so that $f=f_\mathbf{c}$ in the form~\eqref{eq:form} and $P=0$.
But now
\[\log^+|\lambda_{f_\mathbf{c}}(0)|_v=\log^+\left|\prod_{i=1}^{d-1}c_i\right|_v\leq (d-1)\log^+\|\mathbf{c}\|_v=(d-1) g_{\mathrm{crit}, v}(f_{\mathbf{c}}),\]
as claimed.

	Next, let $0\leq x\leq d-1$ be a rational number. If $x=0$, then we can realize $x$ as $\deg(\lambda_f(P))/h_{\mathrm{crit}}(f)$ by taking a non-isotrivial family with a fixed point of multiplier 1. To be concrete, we can take $f_\mathbf{c}$ for $c_1=\cdots=c_{d-2}=t$ and $c_{d-1}=t^{-(d-2)}$, which has critical height $d-1$.

	Otherwise, if $x\neq 0$, let $m\in \ZZ^+$ so that $mx\in\ZZ$, and write $mx=qm+r$, with $q\leq d-1$ and $r<m$, both $q$ and $r$ non-negative integers. For $i\leq q$, set $c_i=t^m$, with $t$ some indeterminate, $c_q=t^r$, and $c_i=1$ for $i>q$ (noting that $q\leq x\leq d-1$). Now, the polynomial $f_\mathbf{c}$ over $\CC(t)$ has a fixed point at $z=0$ with multiplier $\pm \prod_{i=1}^{d-1}c_i=\pm t^{mx}$. On the other hand, Lemma~\ref{lem:gcrit} applied at all places shows that $h_{\mathrm{crit}}(f_\mathbf{c})=m$, and so we realize $x=\deg(\lambda_f(P))/h_{\mathrm{crit}}(f)$ in this example.
 
\end{proof}

We end with three remarks, the first two justifying  claims made  in the introduction, and the last proposing avenues for future work.

\begin{remark}\label{rem:sharp}
Let $d\geq 2$, let $X$ be the rational curve defined by \[X:(d-1)P^{d-1}=dtP^{d-2}+1,\] and let $f(z)=(d-1)z^d-dtz^{d-1}$, so that $f(P)=P$. Since $f$ has critical points at $z=0$ and $z=t$, but the first is fixed,  we compute from the definition that $h_{\mathrm{crit}}(f)=\deg(t)=d-1$. On the other hand, $\lambda_f(P)=f'(P)=d(d-1)P^{d-2}(P-t)$, which has degree $2d-3$ since $\deg(P)=1$. This gives
\[\frac{\deg(\lambda_f(P))}{h_{\mathrm{crit}}(f)}=\frac{2d-3}{d-1}=2+o(1)\]
where $o(1)\to 0$ as $d\to \infty$.

On the other hand, there are infinitely many $t\in X(\CC)$ such that $f_t$ is PCF. To see this, consider the solutions to $f^n_t(t)=0$. On the one hand, if $t$ is a solution to this, then the corresponding specialization $f_t$ is PCF. On the other hand, $f^n_t(t)$ is a polynomial of degree $d^n$ in $t$, in fact the highest-order term in $f^n_t(t)$ is $(-1)^d(d-1)^{(d^n-1)/(d-1)}t^{d^n}$. We have
\[f^{n+1}_t(t)=\left(f^n(t)\right)^{d-1}((d-1)f^n(t)-dt)\]
and so every solution to $f^k_t(t)=0$ for $k\leq n$ is a solution to $f^n(t)=0$. Suppose that the solutions to $f^{n+1}_t(t)=0$ are all already solutions to $f^n_t(t)=0$. Then all solutions to $(d-1)f^n_t(t)-dt=0$ have $f^n_t(t)=0$, and hence $t=0$. In other words, $(d-1)f^n_t(t)-dt=0$ is a polynomial of degree $d^n$, with only $t=0$ as a root, and hence is a constant multiple of $t^{d^n}$. But we can see by  induction that $f^n_t(t)$ is divisible by $t^2$, and so cannot differ by $td/(d-1)$ from a constant multiple of $t^{d^n}$. So there is a root of $f^{n+1}(t)=0$ that is not a root of $f_t^n(t)=0$, and in general for each $n$ there exists a $t_n$ such that the critical points of $f_{t_n}$ consist of a fixed point, and an $n$th preimage of a fixed point. Necessarily these are distinct for distinct $n$, and so there are certainly infinitely many $t\in X(\CC)$ with $f_t$ PCF.
\end{remark}

\begin{remark}\label{rem:contra}
We noted in the introduction that, while a family of cubic polynomials with a generic super-attracting periodic point will have infinitely many PCF specializations, this is not true for polynomials of degree $d\geq 4$. Citing the results of Favre and Gauthier~\cite{fg}, this could be demonstrated by choosing a family with periodic critical point and two infinite, independent critical orbits on the generic fibre. In the interest of specificity, though, we construct a concrete class of examples.

Let $\mathbf{b}\in\PP^{d-2}$, let $t$ be an indeterminate, and consider $f_{t\mathbf{b}}$, i.e., $f_{\mathbf{c}}$ for $c_i=b_it$. Changing the homogeneous coordinates representing $\mathbf{b}$ just rescales the parametrization, but keeps the family the same. By the number field version of Lemma~\ref{lem:gcrit} (essentially~\cite[Lemma~8]{pcfpoly}, but we keep the notation of this note) we have for $t\neq 0$
\begin{align*}
h_{\mathrm{crit}}(f_{t\mathbf{b}})&=\sum_{v\in M_K}\frac{[K_v:\QQ_v]}{[K:\QQ]}g_{\mathrm{crit}, v}(f_{t\mathbf{b}})\\
&\geq  \sum_{v\in M_K}\frac{[K_v:\QQ_v]}{[K:\QQ]}\log\|b_1 t, \ldots, b_{d-1}t\|-O_d(1)\\
&=h_{\PP^{d-2}}(\mathbf{b})-O_d(1).
\end{align*}
In other words, once $h_{\PP^{d-2}}(\mathbf{b})$ is sufficiently large, the non-isotrivial family $f_{t\mathbf{b}}$ specializes to a PCF map only at $t=0$. If we take $h_{\PP^{d-2}}(\mathbf{b})$ large and on a coordinate hyperplane (which we can do once $d\geq 4$), the family will have a generically super-attracting fixed point, and exactly one PCF specialization.
\end{remark}

\begin{remark}
We have considered only the function field case here, but the results in~\cite{pern} can also be used to establish an analogous gap on $h(\lambda_f(P))/h_{\mathrm{crit}}(f)$ in the number field case, albeit with a messier statement, in the case where $f$ has at most one infinite critical orbit up to symmetries. It would be of some interest to determine what the set of possible values of this ratio is in that setting, and how it depends on the number of infinite critical orbits, up to equivalence.
\end{remark}


\begin{thebibliography}{9}
\bibitem{bd} M.~Baker and L.~DeMarco, Special curves and postcritically finite polynomials. \emph{Forum  Math., Pi}, \textbf{1} (2013), e3.
%\bibitem{baker} M.~Baker, A finiteness theorem for canonical heights attached to rational maps over function fields. \emph{J.\ Reine Angew.\ Math} \textbf{626} (2009), pp.~205–233.
%\bibitem{bj} E.~Bedford and M.~Jonsson,  Dynamics of regular polynomial endomorphisms of $\mathbf{C}^k$. \emph{Amer.\ J.\ Math.} \textbf{122} (2000), no.~1, pp.~153--212.
%\bibitem{benedetto} R.~L.~Benedetto, Heights and preperiodic points of polynomials over function fields. \emph{Int.\ Math.\ Res.\ Not.} \textbf{2005} no.~62 (2005), pp.~3855-3866. 
%\bibitem{berteloot} G.~Bassanelli and F.~Berteloot, Bifurcation currents in holomorphic dynamics on $\mathbb{P}^k$. \emph{J.\ Reine Angew.\ Math.} \textbf{608} (2007), pp.~201--235.
%\bibitem{bk} J.~Belk and S.~Koch, Iterated monodromy for a two-dimensional map. \emph{In the tradition of Ahlfors-Bers.\ V, 1–11}, 
%Contemp.\ Math.\ \textbf{510}, AMS, Providence, RI, 2010.
%\bibitem{bijl} R.~Benedetto, P.~Ingram, R.~Jones, and A.~Levy, Attracting cycles in $p$-adic dynamics and height bounds for postcritically finite maps. \emph{Duke Math.\ J.} \textbf{163} (2014), no.~13, pp.~2325--2356.
%\bibitem{bd} J.~Briend and J.~Duval, Deux caract\'{e}risations de la mesure d'\'{e}quilibre d'un endomorphisme de $\mathbf{P}^k(\mathbf{C})$. \emph{Pub.\ Math. de l'IH\'{E}S} \textbf{93}, 145 (2001). 
%\bibitem{bg} E.~Bombieri and W.~Gubler, \emph{Heights in Diophantine Geometry}, volume~4 of \emph{New Mathematical Monographs}. Cambridge University Press, Cambridge, 2006.
%\bibitem{bgs} J.-B. Bost, H.~Gillet,C.~Soul\'{e}, Heights of projective varieties and positive Green forms. \emph{J.\ Amer.\ Math.\ Soc.} \textbf{7} (1994), no.~4, pp.~903-1027. 
%\bibitem{callsilv} G.~Call and J.~H.~Silverman, Canonical heights on varieties with morphisms. \emph{Compositio Math.} \textbf{89} (1993), no.~2, pp.~163-205.
%\bibitem{ch1} Z.~Chatzidakis and E.~Hrushovski, Model theory of difference fields. \emph{Trans.\ Amer.\ Math.\ Soc.} \textbf{351} (1999), pp.~2997–3071.
%\bibitem{ch2} Z.~Chatzidakis and E.~Hrushovski, Difference fields and descent in algebraic dynamics, I. \emph{J.\ IMJ} \textbf{7} no.~4 (2008), pp.~653–686.
%\bibitem{ch3} Z.~Chatzidakis and E.~Hrushovski, Difference fields and descent in algebraic dynamics, II. \emph{J.\ IMJ} \textbf{7} no.~4 (2008), pp.~687–704.
\bibitem{dwy} L.~DeMarco, X.~Wang, and H.~Ye, Bifurcation measures and quadratic rational maps. \emph{Proc.\ Lond.\ Math.\ Soc.} \textbf{111} no.~1 (2015), pp.~149--180. 
%\bibitem{thurston} A. Douady and J. Hubbard, A proof of Thurston’s topological characterization of rational functions. \emph{Acta Math.} \textbf{171} (1993), pp.~263--297.
%\bibitem{dupont} C.~Dupont, Exemples de Latt\`{e}s et domaines faiblement sph\'{e}riques de $\CC^n$. \emph{Manuscripta Math.} \textbf{111} (2003), pp.~357--378.
%\bibitem{favre} C.~Favre, Degeneration of endomorphisms of the complex projective space in the hybrid space. \text{arXiv:1611.08490}.
\bibitem{fgcubic} C.~Favre and T.~Gauthier, Classification of special curves in the space of cubic polynomials. \emph{Int.\ Math.\ Res.\ Not.} \textbf{2018} no.~2 (2018), pp.~362--411.
\bibitem{fg} C.~Favre and T.~Gauthier, \emph{The arithmetic of polynomial dynamical pairs} (arXiv:2004.13801)
%\bibitem{fs} J.E.~Forn\ae ss and N.~Sibony, Critically finite rational maps on $\mathbf{P}^2$. \emph{The Madison Symposium on Complex Analysis (Madison, WI, 1991)}, pp.~245–260, 
%Contemp.\ Math.\ \textbf{137}, AMS, Providence, RI, 1992. 
\bibitem{ght} D.~Ghioca, L.-C.~Hsia, and T.~J.~Tucker, Preperiodic points for families of rational maps. \emph{Proc.\ Lond.\ Math.\ Soc.} \textbf{110} no.~2 (2015), pp.~395--427. 
\bibitem{gkn} D.~Ghioca, H.~Krieger, and K.~D.~Nguyen, A case of the dynamical Andr\'{e}-Oort conjecture. \emph{Int.\ Math.\ Res.\ Not.} \textbf{2016}, no.~3, pp.~738--758.
\bibitem{gkny} D.~Ghioca, H.~Krieger, K.~D.~Nguyen, and H.~Ye,The dynamical Andr\'{e}-Oort conjecture: unicritical polynomials. \emph{Duke Math.\ J.} \textbf{166} no.~1 (2017), pp.~1--25.
\bibitem{gy} D.~Ghioca and H.~Ye, A dynamical variant of the Andr\'{e}-Oort conjecture. \emph{Int.\ Math.\ Res.\ Not.} \textbf{2018} no.~8 (2018), pp.~2447--2480.
\bibitem{var} P.~Ingram, Variation of the canonical height for a family of polynomials \emph{J.\ Reine  Angew.\ Math.}, \textbf{685} (2013), pp.~73--97.
\bibitem{pcfpoly} P~Ingram, A finiteness result for post-critically finite polynomials. \emph{Int.\ Math.\ Res.\ Not.} \textbf{2012} no.~3 (2012), pp.~524--543.
%\bibitem{pcfpn} P.~Ingram, Rigidity and height bounds for certain post-critically finite endomorphisms of $\PP^N$.  \emph{Canadian J.\ Math.} \textbf{68} (2016), pp.~625-654.
\bibitem{hcrit} P.~Ingram, The critical height is a moduli height. \emph{Duke Math.\ J.} \textbf{167}, no.~7 (2018), pp.~1311-1346.
\bibitem{pern} P. Ingram, Critical orbits of polynomials with a periodic point of specified multiplier. \emph{Math.\ Zeit.} \textbf{291} (2019), pp.~1245--1262.
%\bibitem{pcfzar} P.~Ingram, R.~Ramadas, and J.~H.~Silverman, Zariski non-density of certain PCF endomorphisms of $\PP^N$. Preprint: \texttt{arXiv:1910.11290}
%\bibitem{axd} P.~Ingram, Minimally critical regular endomorphisms of $\AA^N$. (in preparation)
%\bibitem{pnheights} P.~Ingram, Explicit heights of divisors relative to endomorphisms of $\PP^N$. (in preparation)
%\bibitem{jonsson} M.~Jonsson, Sums of Lyapunov exponents for some polynomial maps of $\mathbf{C}^2$. \emph{Ergodic Theory Dynam.\ Systems} \textbf{18} (1998), pp.~613--630.
%\bibitem{lang} S.~Lang, \emph{Algebra. Revised third edition.}, volume~211 of \emph{Graduate Texts in Mathematics}. Springer, New York, 2002.
%\bibitem{mahler} K.~Mahler, On some inequalities for polynomials in several variables.  \emph{J.\ London Math.\ Soc.} \textbf{37} (1962) pp.~341-344. 
%\bibitem{milnor} J.~Milnor, On rational maps with two critical points. \emph{Exp.\ Math.} \textbf{9} (2000), Issue 4, pp.~~481--522.
%\bibitem{miller} W.~Miller, The Maximum Order of an Element of a Finite Symmetric Group. \emph{Amer.\ Math.\ Monthly} \textbf{94} (1987), no.~6, pp.~497--506.
%\bibitem{ads} J.~H.~Silverman, \emph{The Arithmetic of Dynamical Systems}, volume~241 of \emph{Graduate Texts in Mathematics}. Springer, New York, 2007.
\bibitem{barbados} J.~H.~Silverman, \emph{Moduli Spaces and Arithmetic Dynamics}, volume~30 of \emph{CRM Monograph Series}. AMS, Providence, 2012.
%\bibitem{zhang} S.W.~Zhang,  Small points and adelic metrics. \emph{J.\ Algebraic Geom.} \textbf{4} (1995), no.~2, pp.~281-300.
\end{thebibliography}
\end{document}